\documentclass[12pt, reqno]{amsart}
\usepackage{amsmath, amsthm, amscd, amsfonts, amssymb, graphicx, color}
\usepackage[bookmarksnumbered, colorlinks, plainpages]{hyperref}
\hypersetup{colorlinks=true,linkcolor=red, anchorcolor=green, citecolor=cyan, urlcolor=red, filecolor=magenta, pdftoolbar=true}

\newtheorem{theorem}{Theorem}[section]
\newtheorem{lemma}[theorem]{Lemma}
\newtheorem{proposition}[theorem]{Proposition}
\newtheorem{corollary}[theorem]{Corollary}
\theoremstyle{definition}

\theoremstyle{remark}

\numberwithin{equation}{section}

\begin{document}
\setcounter{page}{1}

\title[$(m,n)$-Jordan derivations]{Characterizations of $(m,n)$-Jordan derivations on some algebras }

\author[G. An, J. He]{Guangyu An, Jun He}

\address{Department of Mathematics, Shaanxi University of Science and Technology,
Xi'an 710021, China.}
\email{\textcolor[rgb]{0.00,0.00,0.84}{anguangyu310@163.com}}

\address{Department of Mathematics and Physics, Anhui Polytechnic University,
Wuhu 241000, China.}
\email{\textcolor[rgb]{0.00,0.00,0.84}{15121034934@163.com}}

\subjclass[2010]{46L57, 46L51, 46L52.}

\keywords{$(m,n)$-Jordan derivation, $(m,n)$-Jordan derivable mapping, $C^*$-algebra, generalized matrix algebra}

\date{Received: xxxxxx; Revised: yyyyyy; Accepted: zzzzzz.
\newline \indent $^{*}$ Corresponding author}

\begin{abstract}
Let $\mathcal R$ be a ring, $\mathcal{M}$ be a $\mathcal R$-bimodule
and $m,n$ be two fixed nonnegative integers with $m+n\neq0$.
An additive mapping $\delta$ from $\mathcal R$ into $\mathcal{M}$
is called an \emph{$(m,n)$-Jordan derivation} if $(m+n)\delta(A^{2})=2mA\delta(A)+2n\delta(A)A$
for every $A$ in $\mathcal R$.
In this paper, we prove that every $(m,n)$-Jordan
derivation from a $C^{*}$-algebra into its Banach bimodule is zero.
An additive mapping $\delta$ from $\mathcal R$ into $\mathcal{M}$
is called a \emph{$(m,n)$-Jordan derivable mapping} at $W$ in $\mathcal R$ if
$(m+n)\delta(AB+BA)=2m\delta(A)B+2m\delta(B)A+2nA\delta(B)+2nB\delta(A)$
for each $A$ and $B$ in $\mathcal R$ with $AB=BA=W$.
We prove that if $\mathcal{M}$ is a unital $\mathcal A$-bimodule
with a left (right) separating set generated algebraically by all
idempotents in $\mathcal A$, then every $(m,n)$-Jordan derivable mapping at zero from $\mathcal A$ into $\mathcal{M}$
is identical with zero.
We also show that if $\mathcal{A}$ and $\mathcal{B}$ are two unital algebras,
$\mathcal{M}$ is a faithful unital $(\mathcal{A},\mathcal{B})$-bimodule
and $\mathcal{U}={\left[\begin{array}{cc}\mathcal{A} &\mathcal{M} \\\mathcal{N} & \mathcal{B} \\\end{array}\right]}$
is a generalized matrix algebra,
then every $(m,n)$-Jordan derivable mapping at zero from $\mathcal{U}$ into itself is equal to zero.
\end{abstract}\maketitle

\section{Introduction}

Let $\mathcal{R}$ be an associative ring.
For an integer $n\geq2$, $\mathcal{R}$ is said to be \emph{$n$-torsion-free}
if $nA=0$ implies that $A=0$ for every $A$ in $\mathcal{R}$. Recall that a ring $\mathcal{R}$
is \emph{prime} if $A\mathcal{R}B=(0)$ implies that either $A=0$ or $B=0$
for each $A,B$ in $\mathcal{R}$;
and is \emph{semiprime} if $A\mathcal{R}A=(0)$ implies that $A=0$ for every $A$ in $\mathcal{R}$.

Suppose that $\mathcal{M}$ is a $\mathcal{R}$-bimodule.
An additive mapping $\delta$ from $\mathcal{R}$ into $\mathcal{M}$ is
called a \emph{derivation} if
$\delta(AB)=\delta(A)B+A\delta(B)$
for each $A,B$ in $\mathcal{R}$;
and $\delta$ is called a \emph{Jordan derivation} if
$\delta(A^{2})=\delta(A)A+A\delta(A)$
for every $A$ in $\mathcal{R}$.
Obviously, every derivation is a Jordan derivation,
the converse is, in general, not true.
A classical result of Herstein \cite{I. Herstein} proves that every Jordan derivation
on a 2-torsion-free prime ring is a derivation;
Cusack \cite{Cusack} generalizes \cite[Theorem 3.1]{I. Herstein} to
2-torsion-free semiprime rings.

In \cite{M. Bresar J. Vukman},
Bre\v{s}ar and Vukman introduce the concepts of left derivations and Jordan left derivations.
Suppose that $\mathcal{M}$
is a left $\mathcal{R}$-module. An additive mapping $\delta$ from $\mathcal{R}$ into $\mathcal{M}$
is called a \emph{left derivation} if $\delta(AB)=A\delta(B)+B\delta(A)$ for each $A,B$ in $\mathcal{R}$;
and $\delta$ is called a \emph{Jordan left derivation} if
$\delta(A^{2})=2A\delta(A)$
for every $A$ in $\mathcal{R}$. Bre\v{s}ar and Vukman \cite{M. Bresar J. Vukman}
prove that if $\mathcal R$ is a prime ring and $\mathcal{M}$
is a 6-torsion free left $\mathcal{R}$-module,
then the existence of
nonzero Jordan left derivations from $\mathcal R$ into $\mathcal M$
implies that $\mathcal R$ is a commutative ring. Deng \cite{Q. Deng}
shows that \cite[Theorem 2.1]{M. Bresar J. Vukman}
is still true when $\mathcal{M}$ is only 2-torsion free.

In \cite{J. Vukman2}, Vukman introduces the concept of $(m,n)$-Jordan derivations.
Suppose that $\mathcal{M}$ is a $\mathcal{R}$-bimodule and $m,n$ are two fixed nonnegative
integers with $m+n\neq0$.
An additive mapping $\delta$ from $\mathcal{R}$ into $\mathcal{M}$ is
called an \emph{$(m,n)$-Jordan derivation} if
$$(m+n)\delta(A^{2})=2mA\delta(A)+2n\delta(A)A$$
for every $A$ in $\mathcal{R}$.
It is easy to show that the concept of $(m,n)$-Jordan derivations
covers the concept of Jordan derivations as well as the concept
of Jordan left derivations. By Vukman \cite[Theorem 4]{J. Vukman1} and Kosi-Ulbl \cite[Theorem 8]{I. Kosi-Ulbl}, we know that
if $m,n$ are two nonnegative integers with $m\neq n$,
then every $(m,n)$-Jordan derivation from a
complex semisimple Banach algebra into itself is
identically equal to zero.

In Section 2, we prove that
if $m,n$ are two positive integers with $m\neq n$, then every $(m,n)$-Jordan derivation
from a $C^{*}$-algebra into its Banach bimodule is identically equal to zero.

Suppose that $\mathcal M$ is a
$\mathcal R$-bimodule and $m,n$ are two fixed nonnegative
integers with $m+n\neq0$. An additive mapping $\delta$ from $\mathcal{R}$ into $\mathcal{M}$
is called an \emph{$(m,n)$-Jordan derivable mapping} at an element $W$ in $\mathcal{R}$ if
$$(m+n)\delta(AB+BA)=2m\delta(A)B+2m\delta(B)A+2nA\delta(B)+2nB\delta(A)$$
for each $A,B$ in $\mathcal{A}$ with $AB=BA=W$.

Let $\mathcal{J}$ be an ideal in $\mathcal{R}$.
$\mathcal{J}$ is said to be
a \emph{left separating set} of $\mathcal{M}$
if for every $N$ in $\mathcal{M}$,
$N\mathcal{J}=\{0\}$ implies $N=0$;
and $\mathcal{J}$ is said to be
a \emph{right separating set} of $\mathcal{M}$
if for every $M$ in $\mathcal{M}$,
$\mathcal{J}M=\{0\}$ implies $M=0$.
When $\mathcal{J}$ is a left separating set and a right separating set of $\mathcal{M}$,
we call $\mathcal{J}$ a \emph{separating set} of $\mathcal{M}$.
Denote by $\mathfrak{J}(\mathcal{R})$ the subring of $\mathcal R$ generated algebraically by all
idempotents in $\mathcal{R}$.

In Section 3, we assume that $\mathcal A$ is a unital algebra, $\mathcal M$ is a unital $\mathcal A$-bimodule
with a left (right) separating $\mathcal{J}\subseteq\mathfrak{J}(\mathcal{A})$, and $\delta$
is an $(m,n)$-Jordan derivable mapping at zero from $\mathcal A$ into $\mathcal M$ such that $\delta(I)=0$,
we show that if $m,n$ are two positive integers with $m\neq n$, then $\delta$ is identically equal to zero.
As applications, we study the $(m,n)$-Jordan derivable mappings at zero
on some non self-adjoint operator algebras.

A \emph{Morita context} is a set $(\mathcal{A}, \mathcal{B}, \mathcal{M}, \mathcal{N})$
and two mappings $\phi$ and $\varphi$, where $\mathcal{A}$ and $\mathcal{B}$ are two algebras,
$\mathcal{M}$ is an $(\mathcal{A},\mathcal{B})$-bimodule and $\mathcal{N}$ is a $(\mathcal{B},\mathcal{A})$-bimodule,
$\phi:~\mathcal{M}\otimes_{\mathcal{B}}\mathcal{N}\rightarrow\mathcal{A}$
and~$\varphi:~\mathcal{N}\otimes_{\mathcal{A}}\mathcal{M}\rightarrow\mathcal{B}$
are two homomorphisms satisfying the following commutative diagrams:
$$
\CD
   \mathcal{M}\otimes_{\mathcal{B}}\mathcal{N}\otimes_{\mathcal{A}}\mathcal{M} @>\phi\otimes I_{\mathcal{M}}>> \mathcal{A}\otimes_{\mathcal{A}}\mathcal{M} \\
   @V I_{\mathcal{M}}\otimes\varphi VV @V \cong VV  \\
    \mathcal{M}\otimes_{\mathcal{B}}\mathcal{B}@>\cong>> \mathcal{M}
\endCD\\
$$
and
$$
\CD
   \mathcal{N}\otimes_{\mathcal{A}}\mathcal{M}\otimes_{\mathcal{B}}\mathcal{N} @>\varphi\otimes I_{\mathcal{N}}>> \mathcal{B}\otimes_{\mathcal{B}}\mathcal{N} \\
   @V I_{\mathcal{N}}\otimes\phi VV @V \cong VV  \\
    \mathcal{N}\otimes_{\mathcal{A}}\mathcal{A}@>\cong>> \mathcal{N}.
\endCD
$$
These conditions insure that the set
$${\left[\begin{array}{cc}\mathcal{A} &\mathcal{M} \\\mathcal{N} & \mathcal{B} \\\end{array}\right]}
=\left\{{\left[\begin{array}{cc}A & M \\N & B \\\end{array}\right]}
:A\in \mathcal{A}, B\in \mathcal{B}, M\in \mathcal{M}, N\in \mathcal{N}\right\}$$
forms an algebra under the usual matrix addition and the matrix multiplication.
We call it a \emph{generalized matrix algebra}.

Let $\mathcal A, \mathcal B$ be two algebras and $\mathcal M$
be an $(\mathcal A, \mathcal B)$-bimodule, the set
$${\left[\begin{array}{cc}\mathcal{A} &\mathcal{M} \\0 & \mathcal{B} \\\end{array}\right]}
=\left\{{\left[\begin{array}{cc}A & M \\0 & B \\\end{array}\right]}
:A\in \mathcal{A}, B\in \mathcal{B}, M\in \mathcal{M}\right\}$$
under the usual matrix addition and matrix multiplication is called a \emph{triangular algebra}.

$\mathcal{M}$ is called a \emph{left faithful} unital $\mathcal{A}$-module
if for every $A$ in $\mathcal{A}$,~$A\mathcal{M}=\{0\}$ implies $A=0$;
$\mathcal{M}$ is called a \emph{right faithful} unital $\mathcal{B}$-module
if for every $B$ in $\mathcal{B}$, $\mathcal{M}B=\{0\}$ implies $B=0$.
When $\mathcal M$ is a left faithful unital $\mathcal{A}$-module and
a right faithful unital $\mathcal{B}$-module, we call $\mathcal M$ a faithful
unital $(\mathcal{A},\mathcal{B})$-bimodule.

In Section 4, we suppose that $\mathcal{A},\mathcal{B}$ are two unital algebras,
$\mathcal{M}$ is a faithful unital $(\mathcal{A},\mathcal{B})$-bimodule,
and we prove that if $m,n$ are two positive integers with $m\neq n$,
then every $(m,n)$-Jordan derivable mapping at zero from
a generalized matrix algebra $\mathcal{U}={\left[\begin{array}{cc}\mathcal{A} &\mathcal{M} \\\mathcal{N} & \mathcal{B} \\\end{array}\right]}$
into itself is identically equal to zero.

Throughout this paper, $\mathcal{A}$ denotes an algebra over the complex field $\mathbb{C}$, and $\mathcal{M}$
denotes an $\mathcal{A}$-bimodule.

\section{$(m,n)$-Jordan derivations on $C^{*}$-algebras}

In \cite{anguangyu}, we prove that if $n=0$ and $m\neq n$, then every
$(m,n)$-Jordan derivation
from a $C^{*}$-algebra into its Banach bimodule is zero.
In this section, we assume that $m,n$ are two
positive integers with $m\neq n$.
and study the $(m,n)$-Jordan derivations on $C^{*}$-algebras.

The following lemma will be used repeatedly in this section.

\begin{lemma}\cite[Proposition 1]{J. Vukman2}
Let $\mathcal{A}$ be an algebra, $\mathcal{M}$ be an $\mathcal{A}$-bimodule
and $m,n$ be two nonnegative integers with $m\neq n$. If
$\delta$ is an $(m,n)$-Jordan derivation from $\mathcal A$ into $\mathcal M$,
then for each $A,B$ in $\mathcal{A}$, we have that
$$(m+n)\delta(AB+BA)=2m\delta(A)B+2m\delta(B)A+2nA\delta(B)+2nB\delta(A).$$
\end{lemma}

\begin{proposition}
Let $\mathcal{A}$ be a commutative $C^{*}$-algebra, $\mathcal{M}$ be a Banach $\mathcal{A}$-bimodule
and $m,n$ be two positive integers with $m\neq n$. If
$\delta$ is an $(m,n)$-Jordan derivation from $\mathcal{A}$ into $\mathcal{M}$, then $\delta$ is
automatically continuous.
\end{proposition}

\begin{proof}
Let $\mathcal{J}=\{J\in\mathcal{A}:D_{J}(T)=\delta(JT)~\mathrm{is~continuous~for~every}~T~\mathrm{in}~\mathcal{A}\}$.
Since $\mathcal A$ is a commutative algebra and
by Lemma 2.1, we have that
$$nJ\delta(T)+m\delta(T)J=(m+n)\delta(JT)-m\delta(J)T-nT\delta(J)$$
for every $T$ in $\mathcal{A}$ and every $J$ in $\mathcal{J}$. Then
$$\mathcal{J}=\{J\in\mathcal{A}:S_{J}(T)=nJ\delta(T)+m\delta(T)J~\mathrm{is~continuous~every}~T~\mathrm{in}~\mathcal{A}\}.$$

In the following we divide the proof into four steps.

First, we show that $\mathcal{J}$ is a closed two-sided ideal in $\mathcal{A}$. Clearly $\mathcal{J}$
is a right ideal in $\mathcal{A}$. Moreover, for each $A$, $T$ in $\mathcal{A}$ and every $J$ in $\mathcal{J}$, we have that
$$(m+n)\delta(AJT)=m\delta(A)JT+m\delta(JT)A+nA\delta(JT)+nJT\delta(A).$$
Thus $D_{AJ}(T)$ is continuous for every $T$ in $\mathcal{A}$ and $\mathcal{J}$ is also a left ideal in $\mathcal{A}$.

Suppose that $\{J_k\}_{k\geqslant1}\subseteq\mathcal{J}$ and $J\in\mathcal{A}$ such that $\lim\limits_{k \rightarrow \infty}J_k=J$.
Then every $S_{J_{k}}$ is a continuous linear operator; hence we obtain that
$$S_J(T)=nJ\delta(T)+m\delta(T)J=\lim_{k \rightarrow \infty}nJ_k\delta(T)+m\delta(T)J_k=\lim_{k \rightarrow \infty}S_{J_k}(T)$$
for every $T$ in $\mathcal{A}$. By the principle of uniform
boundedness, we have that $S_J$ is norm continuous and
$J\in\mathcal{A}$. Thus, $\mathcal{J}$ is a closed two-sided ideal in
$\mathcal{A}$.

Next, we show that the restriction $\delta|_{\mathcal{J}}$ is norm continuous. Suppose the contrary.
 We can choose $\{J_k\}_{k\geqslant1}\subseteq\mathcal{J}$ such that
 $$\sum_{k=1}^{\infty}{\|J_k\|}^2\leqslant
  1~\mathrm{and}~\|\delta(J_k)\|\rightarrow\infty,~\mathrm{when}~k\rightarrow\infty.$$
Let $B=(\sum_{k=1}^{\infty} J_kJ_k^\ast)^{1/4}$. Then $B$ is a
positive element in $\mathcal{J}$ with $\|B\|\leqslant1$. By \cite[Lemma 1]{Ringrose}
we know that $J_k=BC_k$ for some
$\{C_k\}\subseteq\mathcal{J}$ with $\|C_k\|\leqslant1$, and
$$\|D_B(C_k)\|=\|\delta(BC_k)\|=\|\delta(J_k)\|\rightarrow\infty,~\mathrm{when}~k\rightarrow\infty.$$
This leads to a contradiction. Hence $\delta|_{\mathcal{J}}$ is norm
continuous.

In the following, we show that the $C^{\ast}$-algebra $\mathcal{A}/\mathcal{J}$ is finite-dimensional. Otherwise,
by \cite{T. Ogasawara} we know that $\mathcal{A}/\mathcal{J}$ has an infinite-dimensional abelian $C^{*}$-subalgebra
$\tilde{\mathcal{A}}$. Since the carrier space $X$ of $\tilde{\mathcal{A}}$ is infinite, it follows easily
from the isomorphism
between $\tilde{\mathcal{A}}$ and $C_{0}(X)$ that there is a positive element $H$ in $\tilde{\mathcal{A}}$
whose spectrum is infinite. Hence we can choose some nonnegative continuous mappings $f_{1},~f_{2},\ldots,$
 defined on the positive real axis such that
$$f_jf_k=0~\mathrm{if}~j\neq k~\mathrm{and}~f_j(H)\neq 0~(j=1,~2,\ldots).$$
Let $\varphi$ be a natural mapping from $\mathcal{A}$ into $\mathcal{A}/\mathcal{J}$. Then there exists
a positive element $K$ in $\mathcal{A}$ such that $\varphi(K)=H$. Denote $A_{j}=f_{j}(K)$ for every $j$.
Then we have that $A_{j}\in\mathcal{A}$ and
$$\varphi(A_{j}^{2})=\varphi(f_{j}(K))^{2}=[f_{j}(\varphi(K))]^{2}=f_j(H)^{2}\neq0.$$
It follows that $A_{j}^{2}\notin\mathcal{J}$ and $A_jA_k=0$ if $j\neq k$. If we replace $A_{j}$ by an appropriate scalar multiple, we may suppose that $\|A_{j}\|\leqslant1$.
By $A_j^2\notin\mathcal{J}$, we have that $D_{A_j^{2}}$ is unbounded. Thus, we can choose $T_j\in\mathcal{A}$ such that
$$\|T_j\|\leqslant2^{-j}~\mathrm{and}~(m+n)\|\delta(A_j^2T_j)\|\geqslant (m+n)K\|\delta(A_j)\|+j,$$
where $K=\mathrm{max}\{M,N\}$, $M$ is the bound of the linear mapping
$$(T,M)\rightarrow MT:~\mathcal{A}\times\mathcal{M}\rightarrow\mathcal{A}$$
and $N$ is the bound of the linear mapping
$$(T,M)\rightarrow TM:~\mathcal{A}\times\mathcal{M}\rightarrow\mathcal{A}.$$
Let $C=\sum\limits_{j\geqslant 1}A_jT_j$. Then we have that $\|C\|\leqslant1$ and $A_jC=A_j^2T_j$, and so
\begin{align*}
  \|nA_j\delta(C)+m\delta(C)A_j\| &= \|(m+n)\delta(A_jC)-m\delta(A_j)C-nC\delta(A_j)\| \\
                   &\geqslant\|(m+n)\delta(A_j^2T_j)\|-mM\|\delta(A_j)\|\|C\|-nN\|C\|\|\delta(A_j)\|\\
                   &\geqslant\|(m+n)\delta(A_j^2T_j)\|-(m+n)K\|\delta(A_j)\|\|C\|\\
                   &\geqslant (m+n)K\|\delta(A_j)\|+j-(m+n)K|\delta(A_j)\| =j.
\end{align*}
However, this is impossible because, in fact, $\|A_{j}\|\leqslant1$ and the linear mapping
$$T\rightarrow nT\delta(C)+m\delta(C)T:~\mathcal{A}\rightarrow\mathcal{M}$$
is bounded. Thus we prove that $\mathcal{A}/\mathcal{J}$ is finite-dimensional.

Finally we show that $\delta$ is a continuous linear mapping from $\mathcal A$
into $\mathcal M$.
Since $\mathcal{A}/\mathcal{J}$ is finite-dimensional, we can choose
some elements $A_{1},A_{2},\cdots,A_{r}$ in $\mathcal{A}$
such that $\varphi(A_{1}),\varphi(A_{2}),\cdots,\varphi(A_{r})$ is a basis for the
Banach space $\mathcal{A}/\mathcal{J}$, and let
$\tau_{1},\tau_{2},\cdots,\tau_{r}$ be continuous linear functional on $\mathcal{A}/\mathcal{J}$
such that
$$\tau_{j}(\varphi(A_{k}))=1~\mathrm{when}~j=k~\mathrm{and}~\tau_{j}(\varphi(A_{k}))=0~\mathrm{when}~j\neq k.$$
For every $A$ in $\mathcal A$, we have that
$$\varphi(A)=\sum\limits_{k=1}^{r}c_k\varphi(A_{k}),$$
and the scalars $c_1,c_2,\cdots,c_r$ are determined by $c_j=\tau_{j}(\varphi(\mathcal A))=\rho_{j}(\mathcal A)$,
where $\rho_{j}$ is continuous linear functional $\tau_{j}\circ\varphi$ on $\mathcal{A}$. Since
$$\varphi(A)=\sum\limits_{j=1}^{r}\rho_{j}(A)\varphi(A_{j}),$$
we can obtain that
$\varphi(A-\sum\limits_{j=1}^{r}\rho_{j}(A)A_{j})=0$.
It follows that
$A\rightarrow A-\sum\limits_{j=1}^{r}\rho_{j}(A)A_{j}$
is a continuous mapping from $\mathcal{A}$ into $\mathcal{J}$.
Since $\delta|_{\mathcal{J}}$ is continuous and
$\rho_{1},\rho_{2},\cdots,\rho_{r}$ are continuous, it implies that
$$A\rightarrow[\delta(A)-\sum\limits_{j=1}^{r}\rho_{j}(A)\delta(A_{j})]+\sum\limits_{j=1}^{r}\rho_{j}(A)\delta(A_{j})=\delta(A)$$
is continuous from $\mathcal A$ into $\mathcal M$.
\end{proof}

Given an element $A$ of the algebra $B(\mathcal{H})$ of all bounded linear operators on
a Hilbert space $\mathcal{H}$, we denote by $\mathcal{G}(A)$
the $C^{*}$-algebra generated by $A$. For any self-adjoint subalgebra
$\mathcal{A}$ of $B(\mathcal{H})$, if $\mathcal{G}(B)\subseteq\mathcal{A}$
for every self-adjoint element $B\in\mathcal{A}$, then we call $\mathcal{A}$
\emph{locally closed}. Obviously, every $C^{*}$-algebra is locally closed
and we have the following result.

\begin{lemma}\cite[Corollary 1.2]{Cuntz}
Let $\mathcal{A}$ be a locally closed subalgebra of $B(\mathcal{H})$, $Y$ be a
locally convex linear space and $\psi$ be a linear mapping from $\mathcal{A}$
into $Y$. If $\psi$ is continuous from every commutative self-adjoint subalgebra
of $\mathcal{A}$ into $Y$, then $\psi$ is continuous.
\end{lemma}

By Proposition 2.2 and Lemma 2.3, we can obtain the following corollary.

\begin{corollary}
Let $\mathcal{A}$ be a $C^{*}$-algebra, $\mathcal{M}$ be a Banach $\mathcal{A}$-bimodule
and $m,n$ be two positive integers with $m\neq n$. If
$\delta$ is an $(m,n)$-Jordan derivation from $\mathcal{A}$ into $\mathcal{M}$, then $\delta$ is
automatically continuous.
\end{corollary}

\begin{proof}
By Lemma 2.3, it is sufficient to prove that $\delta$ is continuous from every
commutative self-adjoint subalgebra $\mathcal{B}$ of ${A}$ into $\mathcal{M}$.
It is clear that the norm closure $\bar{\mathcal{B}}$ of $\mathcal{B}$ is a
commutative $C^{*}$-algebra. Thus, we only need to show that the restriction
$\delta|_{\bar{\mathcal{B}}}$ is continuous.

By Proposition 2.2 we know that $\delta|_{\bar{\mathcal{B}}}$ is
automatically continuous. Hence $\delta$ is continuous on
$\mathcal{B}$.
\end{proof}

By Corollary 2.4 and \cite[Theorem 2.3]{anguangyu1}, we have the following theorem immediately.

\begin{theorem}
Let $\mathcal{A}$ be a $C^{*}$-algebra, $\mathcal{M}$ be a Banach $\mathcal{A}$-bimodule
and $m,n$ be two positive integers with $m\neq n$. If
$\delta$ is an $(m,n)$-Jordan derivation from $\mathcal{A}$ into $\mathcal{M}$, then $\delta$ is identically equal to zero.
\end{theorem}

\section{$(m,n)$-Jordan derivable mappings at zero on some algebras}

In \cite{anguangyu4}, we give a characterization of $(m,n)$-Jordan derivable mappings
at zero on some algebras when $n=0$ and $m\neq 0$. In this section,
we assume that $m,n$ are two positive
numbers and study the propositions of $(m,n)$-Jordan derivable mappings at zero.

Let $\mathcal A$ be an algebra. $\mathfrak{J}(\mathcal{A})$ denotes the subalgebra of $\mathcal A$ generated algebraically by all
idempotents in $\mathcal{A}$.

\begin{lemma}\cite[Lemma 2.2]{anguangyu4}
Let $\mathcal A$ be a unital algebra and $\mathcal X$
be a vector space.
If $\phi$ is a bilinear mapping from $\mathcal{A}\times\mathcal{A}$ into $\mathcal{X}$ such that
for each $A,B$ in $\mathcal A$,
$$AB=BA=0\Rightarrow\phi(A,B)=0,$$
then we have that
$$\phi(A,J)+\phi(J,A)=\phi(AJ,I)+\phi(I,JA)$$
for every $A$ in $\mathcal{A}$ and every $J$ in $\mathfrak{J}(\mathcal{A})$.
\end{lemma}

By Lemma 3.1, we have the following result.

\begin{proposition}
Let $\mathcal{A}$ be a unital algebra, $\mathcal{M}$ be a unital $\mathcal{A}$-bimodule
and $m,n$ be two positive integers with $m\neq n$.
If $\delta$ is an $(m,n)$-Jordan derivable mapping at zero from $\mathcal A$ into $\mathcal M$
such that $\delta(I)=0$, then
for every $A$ in $\mathcal{A}$ and every idempotent $P$ in $\mathcal{A}$, we have
the following two statements:\\
$(1)$ $\delta(P)=0;$\\
$(2)$ $\delta(PA)=\delta(AP)=\delta(A)P=P\delta(A).$
\end{proposition}

\begin{proof}
Let $P$ be an idempotent in $\mathcal{A}$, since $P(I-P)=(I-P)P=0$, we have that
$m\delta(P)(I-P)+m\delta(I-P)P+nP\delta(I-P)+n(I-P)\delta(P)=0.$
By $\delta(I)=0$, we can easily show that
\begin{align}
(m+n)\delta(P)-2m\delta(P)P-2nP\delta(P)=0.                                \label{401}
\end{align}
Multiply $P$ from the both sides of \eqref{401} and by $m+n\neq0$, we have that
$P\delta(P)P=0$.
Then multiply $P$ from the left side of \eqref{401} and by $m\neq n$, we can obtain that
$P\delta(P)=0$.
Similarly, we can prove that
$\delta(P)P=0$.
By \eqref{401} and $m+n\neq0$, it is easy to show that
$\delta(P)=0$.

For each $A,B$ in $\mathcal{A}$,
define a bilinear mapping $\phi$ from $\mathcal{A}\times\mathcal{A}$ into $\mathcal M$ by
$$\phi(A,B)=m\delta(A)B+m\delta(B)A+nA\delta(B)+nB\delta(A).$$
It is clear that
$$AB=BA=0\Rightarrow\phi(A,B)=0,$$
by Lemma 3.1, we have that
$$\phi(A,P)+\phi(P,A)=\phi(AP,I)+\phi(I,PA)$$
for every $A$ in $\mathcal{A}$ and every idempotent $P$ in $\mathcal{A}$.
By the definition of $\phi$ and $\delta(I)=0$, it follows that
\begin{align}
(m+n)\delta(AP+PA)=2m\delta(A)P+2nP\delta(A).              \label{307}
\end{align}
Replace $A$ by $AP$ in \eqref{307}, we have that
\begin{align}
(m+n)\delta(AP)+(m+n)\delta(PAP)=2m\delta(AP)P+2nP\delta(AP).              \label{403}
\end{align}
Multiply $P$ from the both sides of \eqref{403} and by $m+n\neq0$, it implies that
$$P\delta(AP)P=P\delta(PAP)P.$$
Similarly, we have that
$$P\delta(PA)P=P\delta(PAP)P.$$
Multiply $P$ from the both sides of \eqref{307}, we can obtain that
\begin{align}
P\delta(A)P=P\delta(AP)P=P\delta(PA)P=P(PAP)P.              \label{306}
\end{align}

Next we prove that $P\delta(A)=P\delta(AP)$, $\delta(A)P=\delta(AP)P$ and $\delta(PA)=\delta(PAP)$.
Replace $A$ by $A-AP$ in \eqref{307}, we have that
\begin{align}
(m+n)\delta(PA-PAP)&=2m\delta(A-AP)P+2nP\delta(A-AP),                  \label{308}
\end{align}
multiply $P$ from the left side of \eqref{308} and by \eqref{306}, we can obtain that
\begin{align}
(m+n)P\delta(PA-PAP)=2nP\delta(A)-2nP\delta(AP),                     \label{309}
\end{align}
replace $A$ by $PA$ in \eqref{309}, we have that
\begin{align}
(m+n)P\delta(PA-PAP)=&2nP\delta(PA)-2nP\delta(PAP)\notag\\
=&2nP\delta(PA-PAP),                                                 \label{310}
\end{align}
by $m\neq n$ and \eqref{310}, it implies that
\begin{align}
P\delta(PA)=P\delta(PAP),                    \label{311}
\end{align}
by \eqref{309} and \eqref{311}, we can obtain that
\begin{align}
P\delta(A)=P\delta(AP).                  \label{312}
\end{align}

Multiply $P$ from the right side of \eqref{308} and by \eqref{306}, it follows that
\begin{align}
(m+n)\delta(PA-PAP)P=2m\delta(A)P-2m\delta(AP)P,                  \label{313}
\end{align}
replace $A$ by $PA$ in \eqref{313}, we have that
\begin{align}
(m+n)\delta(PA-PAP)P=&2m\delta(PA)P-2m\delta(PAP)P\notag\\
=&2m\delta(PA-PAP)P,                                                         \label{314}
\end{align}
by $n\neq m$ and \eqref{314}, we can obtain that
\begin{align}
\delta(PA)P=\delta(PAP)P,                    \label{315}
\end{align}
by \eqref{313} and \eqref{315}, we have that
\begin{align}
\delta(A)P=\delta(AP)P.                    \label{316}
\end{align}
By \eqref{308}, \eqref{312} and \eqref{316}, it follows that
\begin{align}
\delta(PA)=\delta(PAP).                    \label{317}
\end{align}
Similarly, it is easy to obtain three identities as follows:
\begin{align}
P\delta(A)=P\delta(PA),~~\delta(A)P=\delta(PA)P~\mathrm{and}~\delta(AP)=\delta(PAP).                \label{322}
\end{align}

Multiply $P$ from the left side of \eqref{307}, we have that
\begin{align}
(m+n)P\delta(PA+AP)=2mP\delta(A)P+2nP\delta(A),                   \label{328}
\end{align}
by \eqref{312}, \eqref{322} and \eqref{328}, we can obtain that
\begin{align}
P\delta(PA)=P\delta(AP)=P\delta(A)P.                   \label{329}
\end{align}
Multiply $P$ from the right side of \eqref{307}, we have that
\begin{align}
(m+n)\delta(PA+AP)P=2m\delta(A)P+2nP\delta(A)P,                   \label{330}
\end{align}
by \eqref{316}, \eqref{322} and \eqref{330}, it follows that
\begin{align}
\delta(PA)P=\delta(AP)P=P\delta(A)P.                   \label{331}
\end{align}
By \eqref{322}, \eqref{329} and \eqref{331}, we have that
\begin{align}
P\delta(A)=\delta(A)P.                  \label{332}
\end{align}
Finally, by \eqref{307}, \eqref{322} and \eqref{332}, it implies that
$\delta(PA)=\delta(AP)=\delta(A)P=P\delta(A).$
\end{proof}

By the definition of $\mathfrak{J}(\mathcal{A})$ and by Proposition 3.2, it is easy to show the following result.

\begin{corollary}
Let $\mathcal{A}$ be a unital algebra, $\mathcal{M}$ be a unital $\mathcal{A}$-bimodule
and $m,n$ be two positive integers with $m\neq n$.
If $\delta$ is an $(m,n)$-Jordan derivable mapping at zero from $\mathcal A$ into $\mathcal M$
such that $\delta(I)=0$,
then for every $S$ in $\mathfrak{J}(\mathcal{A})$ and every $A$ in $\mathcal{A}$, we have that
$$\delta(SA)=\delta(AS)=\delta(A)S=S\delta(A).$$
\end{corollary}

Recall the definition of separating set.
For an ideal $\mathcal{J}$ of an algebra $\mathcal{A}$, we say that $\mathcal{J}$
is a right separating set of $\mathcal A$-bimodule $\mathcal{M}$
if for every $M$ in $\mathcal{M}$,
$\mathcal{J}M=\{0\}$ implies $M=0$;
and we say that $\mathcal{J}$
is a left separating set of $\mathcal{M}$
if for every $N$ in $\mathcal{M}$,
$N\mathcal{J}=\{0\}$ implies $N=0$.

\begin{theorem}
Let $\mathcal{A}$ be a unital algebra, $\mathcal{M}$ be a unital $\mathcal{A}$-bimodule
with a right or a left separating set $\mathcal J\subseteq\mathfrak{J}(\mathcal{A})$
and $m,n$ be two positive integers with $m\neq n$.
If $\delta$ is an $(m,n)$-Jordan derivable mapping at zero from $\mathcal{A}$ into $\mathcal{M}$
such that $\delta(I)=0$, then $\delta$ is identically equal to zero.
\end{theorem}

\begin{proof}
Let $A, B$ be in $\mathcal{A}$ and every $S$ be in $\mathcal{J}$. By Corollary 3.3, we have that
$$\delta(SAB)=S\delta(AB)$$
and
$$\delta(SAB)=\delta((SA)B)=SA\delta(B).$$
It follows that $S(\delta(AB)-A\delta(B))=0$.
If $\mathcal{J}$ is a right separating set of $\mathcal{M}$, we can obtain that
$\delta(AB)=A\delta(B)$. Take $B=I$ and by $\delta(I)=0$, we have that $\delta(A)=A\delta(I)=0$.

Similarly, if $\mathcal{J}$ is a left separating set of $\mathcal{M}$, then
we also can show that $\delta(A)=\delta(I)A=0$.
\end{proof}

By Theorem 3.4, it is easy to show the following result.

\begin{corollary}
Let $\mathcal{A}$ be a unital algebra with $\mathcal{A}=\mathfrak{J}(\mathcal{A})$,
$\mathcal{M}$ be a unital $\mathcal{A}$-bimodule and $m,n$ be two positive integers with $m\neq n$.
If $\delta$ is an $(m,n)$-Jordan derivable mapping at zero from $\mathcal A$ into $\mathcal M$
such that $\delta(I)=0$, then $\delta$ is identically equal to zero.
\end{corollary}

Let $X$ be a complex Banach space and $B(X)$ be the set of all bounded linear operators on $X$.
In this paper, every subspace of $X$ is a closed linear manifold. By a \emph{subspace lattice} on $X$,
we mean a collection $\mathcal{L}$ of subspaces of $X$ with $(0)$ and $X$ in $\mathcal{L}$
such that, for every family $\{M_{r}\}$ of elements of $\mathcal{L}$, both $\cap M_{r}$
and $\vee M_{r}$ belong to $\mathcal{L}$, where $\vee M_{r}$ denotes the closed linear
span of $\{M_{r}\}$.

For every subspace lattice $\mathcal{L}$ on $X$, we use $\mathrm{Alg}\mathcal{L}$
to denote the algebra of all operators in $B(X)$ that leave members of $\mathcal{L}$ invariant.

For a subspace lattice $\mathcal{L}$ on $X$ and for every $E$ in $\mathcal{L}$, we denote by
$$E_-=\vee\{F\in \mathcal{L}:F\nsupseteq E\},~(0)_-=(0);$$
and
$$E_+=\cap\{F\in \mathcal{L}:F\nsubseteq E\},~X_+=X.$$

A totally ordered subspace lattice $\mathcal{N}$ is called a \emph{nest},
$\mathcal{N}$ is called a \emph{discrete nest} if $L_-\neq L$
for every nontrivial subspace $L$ in $\mathcal{N}$,
and $\mathcal{N}$ is called a \emph{continuous nest} if $L_-=L$
for every subspace $L$ in $\mathcal{N}$.

By \cite{D. Hadwin1} and Theorem 3.4, we have the following two corollaries.

\begin{corollary}
Let $\mathcal{L}$ be a subspace lattice in a von Neumann algebra $\mathcal{B}$
on a Hilbert space $\mathcal{H}$ such that
$\mathcal{H}_-\neq \mathcal{H}$ or $(0)_+\neq(0)$,
and $m,n$ be two positive integers with $m\neq n$.
If $\delta$ is an $(m,n)$-Jordan derivable mapping at zero from $\mathcal{B}\cap\mathrm{Alg}\mathcal{L}$
into $\mathcal B$ such that $\delta(I)=0$, then $\delta$ is identically equal to zero.
\end{corollary}

\begin{corollary}
Let $\mathcal{N}$ be a nest in a von Neumann algebra $\mathcal{B}$
and $m,n$ be two positive integers with $m\neq n$.
If $\delta$ is an $(m,n)$-Jordan derivable mapping at zero from $\mathcal{B}\cap\mathrm{Alg}\mathcal{N}$
into $\mathcal B$ such that $\delta(I)=0$, then $\delta$ is identically equal to zero.
\end{corollary}

Let $\mathcal L$ be a subspace lattice on $X$. Denote by
$\mathcal{J}_\mathcal{L}=\{L\in \mathcal{L}:L\neq (0),~L_-\neq X\}$
and $\mathcal{P}_\mathcal{L}=\{L\in \mathcal{L}:L_-\nsupseteq L\}.$
$\mathcal{L}$ is called a \emph{$\mathcal{J}$-subspace lattice} on $X$
if it satisfies $E\vee E_-=X$ and $E\cap E_-=(0)$ for every $E$ in
$\mathcal{J}_\mathcal{L}$; $\vee\{E:E\in \mathcal{J}_\mathcal{L}\}=X$ and $\cap\{E_-:E\in \mathcal{J}_\mathcal{L}\}=(0)$.
$\mathcal{L}$ is called a \emph{$\mathcal{P}$-subspace lattice} on $X$ if
it satisfies $\vee\{E:E\in\mathcal{P}_\mathcal{L}\}=X$ or $\cap\{E_-:E\in\mathcal{P}_\mathcal{L}\}=(0)$.

The class of $\mathcal{P}$-subspace lattice algebras is very large, it
includes the following: \\
(1) $\mathcal{J}$-subspace lattice algebras;\\
(2) discrete nest algebras;\\
(3) reflexive algebras $\mathrm{Alg}\mathcal{L}$ such that $(0)_{+}\neq(0)$ or $X_{-}\neq X$.

In Y. Chen and J. Li \cite{Y. Chen}, if $\mathcal{L}$ satisfies $\vee\{L:L\in\mathcal{P}_\mathcal{L}\}=X$
or $\cap\{L_-:L\in\mathcal{P}_\mathcal{L}\}=(0)$,
then the ideal~$\mathcal{T}=\mathrm{\mathrm{span}}\{x\otimes f:x\in E,f\in E_{-}^{\perp},E\in\mathcal{P}_\mathcal{L}\}$
in $\mathrm{Alg}\mathcal{L}$ is generated by the idempotents in $\mathrm{Alg}\mathcal{L}$ and $\mathcal{T}$
is right separating set of $B(X)$. It follows the following result.

\begin{corollary}
Let $\mathcal{L}$ be a $\mathcal{P}$-subspace lattice on a Banach space $X$
and $m,n$ be two positive integers with $m\neq n$.
If $\delta$ is an $(m,n)$-Jordan derivable mapping at zero from $\mathrm{Alg}\mathcal{L}$
into $B(X)$ such that $\delta(I)=0$, then $\delta$ is identically equal to zero.
\end{corollary}

Let $\mathcal L$ be a subspace lattice on $X$.
$\mathcal{L}$ is said to be \emph{completely distributive} if its subspaces satisfy the identity
$$\bigwedge_{a\in I}\bigvee_{b\in J}L_{a,b}=\bigvee_{f\in J^{I}}\bigwedge_{a\in I}L_{a,f(a)},$$
where $J^{I}$ denotes the set of all $f:I\rightarrow J$.

Suppose that $\mathcal{L}$ is a completely distributive subspace lattice on $X$
and $\mathcal A=\mathrm{Alg}\mathcal L$.
By D. Hadwin \cite{D. Hadwin}, we know that
$\mathcal{T}=\mathrm{span}\{T:T\in\mathcal{A},\mathrm{rank}~T=1\}$
is an ideal of $\mathcal A$  and by C. Laurie \cite{laurie},
we have that $\mathcal{T}$ is a separating set of $\mathcal{M}$.

\begin{corollary}
Let $\mathcal{L}$ be a completely distributive subspace lattice
on a Hilbert space $\mathcal{H}$,
$\mathcal{M}$ be a dual normal Banach $\mathrm{Alg}\mathcal{L}$-bimodule
and $m,n$ be two positive integers with $m\neq n$.
If $\delta$ is an $(m,n)$-Jordan derivable mapping at zero from $\mathrm{Alg}\mathcal{L}$
into $\mathcal M$ such that $\delta(I)=0$, then $\delta$ is identically equal to zero.
\end{corollary}

\begin{corollary}
Let $\mathcal{A}$ be a unital subalgebra of $B(X)$ such that $\mathcal{A}$ contains $\{x_{0}\otimes f:f\in X^{*}\}$,
where $0\neq x_{0}\in X$, and $m,n$ be two positive integers with $m\neq n$.
If $\delta$ is an $(m,n)$-Jordan derivable mapping at zero from $\mathcal{A}$
into $B(X)$ such that $\delta(I)=0$, then $\delta$ is identically equal to zero.
\end{corollary}

Similar to the corollary 3.10, we have the following result.

\begin{corollary}
Let $\mathcal{A}$ be a unital subalgebra of $B(X)$ such that $\mathcal{A}$ contains $\{x\otimes f_{0}:x\in X\}$,
where $0\neq f_{0}\in X^{*}$, and $m,n$ be two positive integers with $m\neq n$.
If $\delta$ is an $(m,n)$-Jordan derivable mapping at zero from $\mathcal{A}$
into $B(X)$ such that $\delta(I)=0$, then $\delta$ is identically equal to zero.
\end{corollary}

\section{$(m,n)$-Jordan derivable mappings at zero on generalized matrix algebras}

In this section, we give a characterization of $(m,n)$-Jordan derivable mappings at zero on generalized matrix algebras.

\begin{theorem}
Suppose that $\mathcal{A},\mathcal{B}$ are two unital algebras,
$m,n$ be two positive integers with $m\neq n$, and
$\mathcal{U}={\left[\begin{array}{cc}\mathcal{A} &\mathcal{M} \\\mathcal{N} & \mathcal{B} \\\end{array}\right]}$
is a generalized matrix ring. If one of the following four statements holds:\\
$(1)$~$\mathcal{M}$ is a faithful unital $(\mathcal{A},\mathcal{B})$-bimodule;\\
$(2)$~$\mathcal{N}$ is a faithful unital $(\mathcal{B},\mathcal{A})$-bimodule;\\
$(3)$~$\mathcal{M}$ is a faithful unital left $\mathcal{A}$-module, $\mathcal{N}$ is a faithful unital left $\mathcal{B}$-module;\\
$(4)$~$\mathcal{N}$ is a faithful unital right $\mathcal{A}$-module, $\mathcal{M}$ is a faithful unital right $\mathcal{B}$-module,\\
then every $(m,n)$-Jordan derivable mapping from generalized matrix algebra $\mathcal{U}$ into itself
satisfies $\delta({\left[\begin{array}{cc}I_\mathcal{A} &0 \\0 & I_\mathcal{B} \\\end{array}\right]})={\left[\begin{array}{cc}0 &0 \\0 & 0 \\\end{array}\right]}$ is identically equal to zero.
\end{theorem}

\begin{proof}
Since $\delta$ is a linear mapping, for each $A\in\mathcal A$, $B\in\mathcal B$, $M\in\mathcal M$ and $N\in\mathcal N$,
we have that
\begin{align*}
&\delta\left({\left[\begin{array}{cc}A & M \\N & B \\\end{array}\right]}\right)\\
=&\left[\begin{array}{cc}a_{11}(A)+b_{11}(M)+c_{11}(N)+d_{11}(B) &a_{12}(A)+b_{12}(M)+c_{12}(N)+d_{12}(B) \\a_{21}(A)+b_{21}(M)+c_{21}(N)+d_{21}(B) & a_{22}(A)+b_{22}(M)+c_{22}(N)+d_{22}(B) \\\end{array}\right],
\end{align*}
where $a_{ij},b_{ij},c_{ij}~\mathrm{and}~d_{ij}$ are linear mappings, $i,j\in\{1,2\}$.

Let $I_{\mathcal{A}}$ be a unit element in $\mathcal{A}$ and $I_{\mathcal{B}}$ be a unit element in $\mathcal{B}$.
For every $M$ in $\mathcal{M}$, suppose that
$T=\left[\begin{array}{cc}0 & M \\0 & 0 \\\end{array}\right]$ and $S=\left[\begin{array}{cc}I_{\mathcal{A}} & 0 \\0 & 0 \\\end{array}\right]$.
By Proposition 3.2, we have that $\delta(TS)=\delta(ST)$, that is
\begin{align*}
\delta\left(\left[\begin{array}{cc}0 & M \\0 & 0 \\\end{array}\right]\left[\begin{array}{cc}I_{\mathcal{A}} & 0 \\0 & 0 \\\end{array}\right]\right)
=\delta\left(\left[\begin{array}{cc}I_{\mathcal{A}} & 0 \\0 & 0 \\\end{array}\right]\left[\begin{array}{cc}0 & M \\0 & 0 \\\end{array}\right]\right),
\end{align*}
it follows that
\begin{align*}
0=\delta\left(\left[\begin{array}{cc}0 & M \\0 & 0 \\\end{array}\right]\right)
=\left[\begin{array}{cc}b_{11}(M)&b_{12}(M)\\b_{21}(M)&b_{22}(M)\\\end{array}\right].
\end{align*}
Thus, for every $M$ in $\mathcal{M}$, we can obtain that
\begin{align}
b_{11}(M)=b_{12}(M)=b_{21}(M)=b_{22}(M)=0.                                   \label{334}
\end{align}
Similarly, for every $N$ in $\mathcal{N}$, we can show that
\begin{align}
c_{11}(N)=c_{12}(N)=c_{21}(N)=c_{22}(N)=0.                                  \label{335}
\end{align}

For every $A$ in $\mathcal{A}$,
suppose that $T=\left[\begin{array}{cc}A & 0 \\0 & 0 \\\end{array}\right]$ and $S=\left[\begin{array}{cc}I_{\mathcal{A}} & 0\\0 & 0 \\\end{array}\right]$.
By Proposition 3.2, we have that $\delta(TS)=S\delta(T)=\delta(T)S$, that is
\begin{align*}
\delta\left(\left[\begin{array}{cc}A & 0 \\0 & 0 \\\end{array}\right]\left[\begin{array}{cc}I_{\mathcal{A}} & 0 \\0 & 0 \\\end{array}\right]\right)
=\left[\begin{array}{cc}I_{\mathcal{A}} & 0 \\0 & 0 \\\end{array}\right]\delta\left(\left[\begin{array}{cc}A & 0 \\0 & 0 \\\end{array}\right]\right)
=\delta\left(\left[\begin{array}{cc}A & 0 \\0 & 0 \\\end{array}\right]\right)\left[\begin{array}{cc}I_{\mathcal{A}} & 0 \\0 & 0 \\\end{array}\right],
\end{align*}
it follows that
\begin{align*}
\left[\begin{array}{cc}a_{11}(A) & a_{12}(A) \\a_{21}(A) & a_{22}(A)\\\end{array}\right]
=\left[\begin{array}{cc}I_{\mathcal{A}} & 0 \\0 & 0 \\\end{array}\right]\left[\begin{array}{cc}a_{11}(A) & a_{12}(A) \\a_{21}(A) & a_{22}(A)\\\end{array}\right]
=\left[\begin{array}{cc}a_{11}(A) & a_{12}(A) \\a_{21}(A) & a_{22}(A)\\\end{array}\right]\left[\begin{array}{cc}I_{\mathcal{A}} & 0 \\0 & 0 \\\end{array}\right],
\end{align*}
it implies that
\begin{align*}
\left[\begin{array}{cc}a_{11}(A) & a_{12}(A) \\a_{21}(A) & a_{22}(A)\\\end{array}\right]
=\left[\begin{array}{cc}a_{11}(A) & a_{12}(A) \\0 & 0\\\end{array}\right]
=\left[\begin{array}{cc}a_{11}(A) &0 \\a_{21}(A) & 0\\\end{array}\right].
\end{align*}
Thus, for every $A$ in $\mathcal{A}$, we can obtain that
\begin{align}
a_{12}(A)=a_{21}(A)=a_{22}(A)=0.                                  \label{336}
\end{align}
Similarly for every $B$ in $\mathcal{B}$, we can show that
\begin{align}
d_{12}(B)=d_{21}(B)=d_{11}(B)=0.                                  \label{337}
\end{align}

In the following, we prove that $a_{11}(A)=d_{22}(B)=0$.

Suppose that $\mathcal{M}$ is a faithful unital $(\mathcal{A},\mathcal{B})$-bimodule.
For every $A$ in $\mathcal{A}$, suppose that
$T=\left[\begin{array}{cc}A & -AM \\0 & 0 \\\end{array}\right]$ and $S=\left[\begin{array}{cc}0 & M\\0 & I_{\mathcal{B}} \\\end{array}\right]$,
it is clear that $TS=ST=0$. By the definition of $\delta$, we have that
$$m\delta(T)S+m\delta(S)T+nT\delta(S)+nS\delta(T)=0.$$
By \eqref{334} and \eqref{336}, we can obtain that
\begin{align*}
m\left[\begin{array}{cc}a_{11}(A) & 0 \\0 & 0 \\\end{array}\right]\left[\begin{array}{cc}0 & M\\0 & I_{\mathcal{B}} \\\end{array}\right]
+n\left[\begin{array}{cc}0 & M\\0 & I_{\mathcal{B}} \\\end{array}\right]\left[\begin{array}{cc}a_{11}(A) & 0 \\0 & 0 \\\end{array}\right]
=0.
\end{align*}
It follows that
$m\left[\begin{array}{cc}0 & a_{11}(A)M \\0 & 0\\\end{array}\right]=0$,
it means that $ma_{11}(A)M$ for every $M$ in $\mathcal M$,
since $m>0$ and $\mathcal M$ is a left faithful unital $\mathcal A$-module,
we can obtain that $a_{11}(A)=0$.
Similarly, since $n>0$ and $\mathcal M$ is a right faithful unital $\mathcal B$-module,
we can obtain that $d_{22}(B)=0$.

Suppose that $\mathcal{N}$ is a faithful unital $(\mathcal{B},\mathcal{A})$-bimodule.
Similar to the above method, we also can prove that $a_{11}(A)=d_{22}(B)=0$.

Suppose that $\mathcal{M}$ is a faithful unital left $\mathcal{A}$-module and
$\mathcal{N}$ is a faithful unital left $\mathcal{B}$-module.
For every $A$ in $\mathcal A$, we have prove that $a_{11}(A)=0$.
For every $B$ in $\mathcal{B}$, suppose that
$T=\left[\begin{array}{cc}0 & 0 \\-NB & B \\\end{array}\right]$ and $S=\left[\begin{array}{cc}I_{\mathcal{A}} & 0\\N & 0 \\\end{array}\right]$,
it is clear that $TS=ST=0$. By the definition of $\delta$, we have that
$$m\delta(T)S+m\delta(S)T+nT\delta(S)+nS\delta(T)=0.$$
By \eqref{334} and \eqref{336}, we can obtain that
\begin{align*}
m\left[\begin{array}{cc}0 & 0 \\0 & d_{22}(B) \\\end{array}\right]\left[\begin{array}{cc}I_{\mathcal{A}} & 0\\N & 0 \\\end{array}\right]
+n\left[\begin{array}{cc}I_{\mathcal{A}} & 0\\N & 0 \\\end{array}\right]\left[\begin{array}{cc}0 & 0 \\0 & d_{22}(B) \\\end{array}\right]
=0.
\end{align*}
It follows that
$m\left[\begin{array}{cc}0 & 0 \\0 & d_{22}(B)N\\\end{array}\right]=0$,
it means that $md_{22}(B)N$ for every $N$ in $\mathcal N$,
since $m>0$ and $\mathcal N$ is a left faithful unital $\mathcal N$-module,
we can obtain that $d_{22}(B)=0$.
\end{proof}

\bibliographystyle{amsplain}

\end{document}